\documentclass{article}
\usepackage{amsmath,amssymb,amsthm,graphicx,epsfig,subfigure,float,url}
\usepackage[colorlinks=true]{hyperref}
\usepackage{pdfsync}

\usepackage{verbatim}
\usepackage[applemac]{inputenc}

\def\R{\mathrm{I\kern-0.21emR}}
\def\N{\mathrm{I\kern-0.21emN}}

\renewcommand{\geq}{\geqslant}
\renewcommand{\leq}{\leqslant}

\newtheorem{theorem}{Theorem}

\newtheorem{corollary}{Corollary}

\newtheorem{lemma}{Lemma}
\theoremstyle{definition}
\theoremstyle{definition}\newtheorem{remark}{Remark}

\title{Finite-dimensional predictor-based feedback stabilization of a 1D linear reaction-diffusion equation with boundary input delay}

\author{Delphine Bresch-Pietri \and Christophe Prieur\footnote{Control Systems Department, GIPSA-lab, 11 rue des Math\'ematiques, BP 46, 38402 Saint-Martin d'H\`eres,
France ({\tt Christophe.Prieur@gipsa-lab.fr}).}
\and
Emmanuel Tr\'elat\footnote{Universit\'e Pierre et Marie Curie (Univ. Paris 6) and Institut Universitaire de France and Team GECO Inria Saclay,
CNRS UMR 7598, Laboratoire Jacques-Louis Lions, F-75005, Paris, France (\texttt{emmanuel.trelat@upmc.fr}).}
}

\date{Preliminary version}

\begin{document}
\maketitle

\begin{abstract}
We consider a one-dimensional controlled reaction-diffusion equation, where the control acts on the boundary and is subject to a constant delay. Such a model is a paradigm for more general parabolic systems coupled with a transport equation. We prove that this is possible to stabilize (in $H^1$ norm) this process by means of an explicit predictor-based feedback control that is designed from a finite-dimensional subsystem. The implementation is very simple and efficient and is based on standard tools of pole-shifting.
Our feedback acts on the system as a finite-dimensional predictor.
We compare our approach with the backstepping method.

\end{abstract}

\section{Introduction and main result}\label{sec_intro}
Let $L>0$ and let $c\in L^\infty(0,L)$.
We consider the 1D heat equation on $(0,L)$ with a delayed boundary control
\begin{equation} \label{eqcont0}
\begin{split}
&y_t = y_{xx}+c(x)y,  \\
&y(t,0)=0,\ y(t,L)=u_D(t)=u(t-D),
\end{split}
\end{equation}
where the state is $y(t,\cdot):[0,L]\rightarrow\R$ and the control is $u_D(t) = u(t-D)$, with $D> 0$ a constant delay. 

Our objective is to design a feedback control stabilizing \eqref{eqcont0}.

%
%

\medskip

There have been a number of works in the literature dealing with the stabilization of processes with input delays
but only few contributions do exist for processes driven by PDE's. 
The academic problem that we investigate here has been studied in \cite{Krstic_SCL2009} with a backstepping approach. 

\medskip

To be more precise with initial conditions, we assume that we are only interested in what happens for $t\geq 0$. We consider an initial condition
$$
y(0,\cdot)=y_0(\cdot) \in L^2(0,L),
$$
and since the boundary control is retarded with the delay $D$, we assume that no control is applied within the time interval $(0,D)$. In other words, we assume that $u_D(t)=0$ for every $t\in(0,D)$.
For every $t>D$ on, a nontrivial control $u_D(t)$ can then be applied. In what follows we are going to design a feedback control whose value $u(t-D)$ only depends on the values of $X_1(s)$ with $0<s<t-D$.

Our strategy begins with a spectral analysis of the operator underlying the control system \eqref{eqcont0} which is split in two parts. The first part of the system is finite dimensional and contains all unstable modes, whereas the second part is infinite dimensional and contains all stable modes. The design of our feedback is realized on the finite-dimensional part of the system. We use the Artstein model reduction and then design a Kalman gain matrix in a standard way with the pole-shifting theorem. Then we invert the Artstein transform and end up with the desired feedback.
We stress that the feedback that we design in such a way is very easy to implement in practice. We first show that this feedback stabilizes exponentially the finite-dimensional part of the system, and then, using an appropriate Lyapunov function, we prove that it stabilizes as well the whole system. 
Note that the exponential asymptotic stability result holds true for every possible value of the delay $D\geq 0$.

\begin{theorem}\label{thm1}
The equation \eqref{eqcont0} with boundary inpu delay is exponentially stabilizable, with a feedback that is built with a finite-dimensional linear control system with input delay. More precisely, with this feedback the function $t\mapsto\Vert y(t,\cdot)\Vert_{H^1(0,L)}$ converges exponentially to $0$ as $t$ tends to $+\infty$.
\end{theorem}

\section{Construction of the feedback and proof of Theorem \ref{thm1}}
\subsection{Spectral reduction}
First of all, in order to deal rather with a homogeneous Dirichlet problem (which is more convenient), we set
\begin{equation}\label{defw}
w(t,x)=y(t,x)-\frac{x}{L}u_D(t),
\end{equation}
and we suppose that the control $u_D$ is derivable for all positive times (this will be true in the construction that we will carry out). This leads to
\begin{equation}\label{reducedproblem2}
\begin{split}
& w_t=w_{xx}+cw+\frac{x}{L}cu_D-\frac{x}{L}u_D' ,  \quad \forall t>0, \;\forall x \in (0,1) , \\
& w(t,0)=w(t,L)=0, \quad \forall t>0, \\
& w(0,x)=y(0,x)-\frac{x}{L}u_D(0) , \quad \forall x \in (0,1).
\end{split}
\end{equation}
We define the operator
\begin{equation}
A=\partial_{xx}+c(\cdot)\mathrm{id},\ 
\end{equation}
on the domain $D(A)=H^2(0,L)\cap H^1_0(0,L)$. 
Then the above control system is
\begin{equation}\label{newzero}
w_t(t,\cdot)=Aw(t,\cdot)+a(\cdot)u_D(t)+b(\cdot)u_D'(t),
\end{equation}
with $a(x)=\frac{x}{L}c(x)$ and $b(x)=-\frac{x}{L}$ for every $x\in(0,L)$.

Noting that $A$ is self-adjoint and of compact inverse, we consider a Hilbert basis $(e_j)_{j\geq 1}$ of $L^2(0,L)$ consisting of eigenfunctions of $A$, associated with the sequence of eigenvalues $(\lambda_j)_{j\geq 1}$. Note that
$$-\infty<\cdots<\lambda_j<\cdots<\lambda_1\quad \textrm{and}\quad \lambda_j\underset{j\rightarrow +\infty}{\longrightarrow}-\infty ,$$
and that $e_j(\cdot)\in H^1_0(0,L)\cap C^2([0,L])$ for every $j\geq 1$.
Every solution $w(t,\cdot)\in H^2(0,L)\cap H^1_0(0,L)$ of \eqref{newzero} can be expanded as a series in the eigenfunctions $e_j(\cdot)$, convergent in $H_0^1(0,L)$,
$$w(t,\cdot)=\sum_{j=1}^{\infty}w_j(t)e_j(\cdot),$$
and one gets the infinite-dimensional control system
\begin{equation}\label{eq16}
w_j'(t)=\lambda_jw_j(t)+a_ju_D(t)+b_ju_D'(t),
\end{equation}
with
\begin{equation}
\begin{split}
a_j&=\left\langle a(\cdot),e_j(\cdot)\right\rangle_{L^2(0,L)} = \frac{1}{L}\int_0^L xc(x)e_j(x) dx, \\
b_j&=\langle b(\cdot),e_j(\cdot)\rangle_{L^2(0,L)} = -\frac{1}{L}\int_0^L xe_j(x)dx,
\end{split}
\end{equation}
for every $j\in\N^*$.
We define
\begin{equation}\label{eqalphaD}
\alpha_D(t)=u_D'(t),
\end{equation}
and we consider from now on $u_D(t)$ as a state and $\alpha_D(t)$ as a control (destinated to be a delayed feedback, with constant delay $D$), so that equations \eqref{eq16} and \eqref{eqalphaD} form an infinite-dimensional control system controlled by $\alpha_D$, written as
\begin{equation}\label{sys-dim-infinie}
\begin{split}
u_D'(t) &= \alpha_D(t), \\
w_1'(t) &=\lambda_1w_1(t)+a_1u_D(t)+b_j\alpha_D(t), \\
& \vdots \\
w_j'(t) &=\lambda_jw_j(t)+a_ju_D(t)+b_j\alpha_D(t), \\
& \vdots 
\end{split}
\end{equation}

Let $n$ be the number of nonnegative eigenvalues, and let $\eta>0$ be such that
\begin{equation}\label{refeta}
\forall k>n\quad \lambda_k<-\eta<0.
\end{equation}
Let $\pi_1$ be the orthogonal projection onto the subspace of $L^2(0,L)$ spanned by
$e_1(\cdot),\ldots,e_n(\cdot)$, and let
\begin{equation}\label{newun}
w^1(t)=\pi_1w(t,\cdot)=\sum_{j=1}^nw_j(t)e_j(\cdot).
\end{equation}
With the matrix notations
\begin{equation}\label{eq:def:A1}
X_1(t)=\begin{pmatrix} u_D(t) \\ w_1(t) \\ \vdots \\ w_n(t)
\end{pmatrix} , \
A_1=\begin{pmatrix}
0         &       0         & \cdots &    0           \\
a_1 & \lambda_1 & \cdots &    0           \\
\vdots    &  \vdots         & \ddots &   \vdots       \\
a_n &  0              & \cdots & \lambda_n
\end{pmatrix} , \
B_1=\begin{pmatrix} 1 \\ b_1 \\ \vdots \\
b_n \end{pmatrix} ,
\end{equation}
the $n$ first equations of 
\eqref{sys-dim-infinie}
form the finite-dimensional control system with input delay
\begin{equation}\label{systfini}
X'_1(t)=A_1X_1(t) + B_1\alpha_D(t) = A_1X_1(t) + B_1\alpha(t-D).
\end{equation}
Note that the state $X_1(t)\in\R^{n+1}$ involves the term $u_D(t)$ which is destinated to be delayed. 

Our objective is to design a feedback control $\alpha$ exponentially stabilizing the infinite-dimensional system \eqref{sys-dim-infinie}. 
We follows an idea used in \cite{CoronTrelat2004,CoronTrelat2006} in order to stabilize nonlinear heat and wave equations around a steady-state.
The idea consists of first designing a feedback control exponentially stabilizing the finite-dimensional system \eqref{systfini}, and then of proving that this feedback actually stabilizes the whole system \eqref{sys-dim-infinie}. The idea underneath is that the finite-dimensional system \eqref{systfini} contains the unstable modes of the whole system \eqref{sys-dim-infinie}, and thus has to be stabilized. It is however not obvious that this feedback stabilizing the unstable finite-dimensional part actually stabilizes as well the whole system (\ref{sys-dim-infinie}),  and this is proved using an appropriate Lyapunov functional.

Before going into details, we stress that this stabilization procedure is carried out with a very simple approach, easy to implement, and using very classical and well-known tools from the finite-dimensional linear setting.

\subsection{Stabilization of the unstable finite-dimensional part}
Let us design a feedback control stabilizing the control system with input delay \eqref{systfini}, as well as a Lyapunov functional.
First of all, following the so-called Artstein model reduction (see \cite{Artstein,Richard2003}), we set, for every $t\in\R$,
\begin{equation}\label{transfo_Artstein}
Z_1(t) = X_1(t) + \int_{t-D}^t e^{(t-s-D)A_1}B_1\alpha(s)\, ds = X_1(t) + \int_{0}^D e^{-\tau A_1}B_1\alpha(t-D+\tau)\, d\tau  ,
\end{equation}
and we get immediately
\begin{equation}\label{systfiniArtstein}
\dot{Z}_1(t) = A_1Z_1(t)+e^{-DA_1}B_1\alpha(t),
\end{equation}
which is a usual linear control system, without input delay, in $\R^{n+1}$.

\begin{lemma}\label{lem:1}
For every $D\geq 0$, the pair $(A_1,e^{-DA_1}B_1)$ satisfies the Kalman condition, that is,
\begin{equation}\label{kalman}
\mathrm{rank} \left( e^{-DA_1}B_1, A_1e^{-DA_1}B_1, \ldots, A_1^{n}e^{-DA_1}B_1  \right) = n+1 .
\end{equation}
\end{lemma}

\begin{proof}
Since $A_1$ and $e^{-DA_1}$ commute, and since $e^{-DA_1}$ is invertible, we have
\begin{equation*}
\begin{split}
& \mathrm{rank} \left( e^{-DA_1}B_1, A_1e^{-DA_1}B_1, \ldots, A_1^{n}e^{-DA_1}B_1  \right) \\
=&\ \mathrm{rank} \left( e^{-DA_1}B_1, e^{-DA_1}A_1B_1, \ldots, e^{-DA_1}A_1^{n}B_1  \right) \\
=&\ \mathrm{rank} \left( B_1, A_1B_1, \ldots, A_1^{n}B_1  \right),
\end{split}
\end{equation*}
and hence it suffices to prove that the pair $(A_1,B_1)$ satisfies the Kalman condition.
A simple computation leads to
\begin{equation}\label{deter1}
\textrm{det} \left( B_1, A_1B_1, \ldots, A_1^{n}B_1 \right)
= \prod_{j=1}^{n}(a_j+\lambda_jb_j)\,\mathrm{VdM}(\lambda_1,\ldots,\lambda_n) ,
\end{equation}
where $\mathrm{VdM}(\lambda_1,\ldots,\lambda_n)$ is a Van der Monde
determinant, and thus is never equal to zero since the real numbers $\lambda_j$, $j=1\ldots n$, are all distinct.
On the other part, using the fact that every $e_j(\cdot)$ is an eigenfunction of $A$ and belongs to $H^1_0(0,L)$, we have, for every integer $j$,
\begin{equation*}
a_j+\lambda_jb_j
= \frac{1}{L} \int_0^L x\left( c(x)e_j(x) -\lambda_je_j(x) \right) dx  
= -\frac{1}{L} \int_0^L x e_j''(x) \, dx  
= - e_j'(L), 
\end{equation*}
which is not equal to zero since $e_j(L)=0$ and $e_j(\cdot)$ is a nontrivial solution of a linear second-order scalar differential equation. The lemma is proved.
\end{proof}

Since the control system \eqref{systfiniArtstein} satisfies the Kalman condition, the well-known pole-shifting theorem and Lyapunov theorem imply the existence of a stabilizing gain matrix and of a Lyapunov functional (see, e.g., \cite{Khalil,Trelat2005}). This yields the following corollary.

\begin{corollary}\label{corkalman}
For every $D\geq 0$ there exists a $1\times (n+1)$ matrix $K_1(D)=\begin{pmatrix} k_0(D),k_1(D),\ldots,k_n(D) \end{pmatrix} $ such that $A_1+B_1e^{-DA_1}K_1(D)$ admits $-1$ as an eigenvalue with order $n+1$.
Moreover there exists a $(n+1)\times (n+1)$ symmetric positive definite matrix $P(D)$ such that
\begin{equation}\label{poleshifting}
P(D)\left(A_1+B_1e^{-DA_1}K_1(D)\right)+ \left(A_1+e^{-DA_1}B_1K_1(D)\right)^\top P(D) = -I_{n+1} .
\end{equation}
In particular, the function
\begin{equation}\label{defV1}
V_1(Z_1) = \frac{1}{2} Z_1^\top P(D)Z_1
\end{equation}
is a Lyapunov function for the closed-loop system $\dot{Z}_1(t) = (A_1+e^{-DA_1}B_1K_1(D))Z_1(t)$.
\end{corollary}

\begin{remark}
It can even be proved that $K_1(D)$ and $P(D)$ are smooth (i.e., of class $C^\infty$) with respect to $D$, but we do not need this property in this paper.
\end{remark}

\begin{remark} 
In the statement above we chose $-1$ as an eigenvalue of $A_1+B_1e^{-DA_1}K_1(D)$, but actually the pole-shifting theorem implies that, for every $(n+1)$-tuple $(\mu_0,\ldots,\mu_n)$ of eigenvalues there exists a $1\times (n+1)$ matrix $K_1(D)$ such that the eigenvalues $A_1+B_1e^{-DA_1}K_1(D)$ are exactly $(\mu_0,\ldots,\mu_n)$. The eigenvalue $-1$ was chosen for simplicity. What is important is to ensure that $A_1+B_1e^{-DA_1}K_1(D)$ is a Hurwitz matrix (that is, whose eigenvalues have a negative real part).

In practice other choices can be done, which can be more efficient according to such or such criterion (see \cite{Trelat2005}). For instance, instead of using the pole-shifting theorem, one could design a stabilizing gain matrix $K_1$ by using a standard Riccati procedure.
\end{remark}

\begin{remark}\label{rem_Lyap1}
From Corollary \ref{corkalman} we infer that for every $D\geq 0$ there exists $C_1(D)>0$ (depending smoothly on $D$) such that
\begin{equation}
\frac{d}{dt}V_1(Z_1(t)) = -\Vert Z_1(t)\Vert_{\R^{n+1}}^2 \leq -C_1(D)\, V_1(Z_1(t)),
\end{equation}
where $\Vert\ \Vert_{\R^{n+1}}$ is the usual Euclidean norm in $\R^{n+1}$.
\end{remark}

From Corollary \ref{corkalman}, the feedback $\alpha(t)=K_1(D)Z_1(t)$ stabilizes exponentially the control system \eqref{systfiniArtstein}.
Since $\alpha(t-D)$ is used in the control system \eqref{systfini}, and since in general we are only concerned with prescribing the future of a system, starting at time $0$, we assume that the control system \eqref{systfini} is uncontrolled for $t<0$, and from the starting time $t=0$ on we let the feedback act on the system. In other words, we set
\begin{equation}\label{defalpha}
\alpha(t) = \left\{ \begin{array}{ll}
0 & \textrm{if}\ t<D,\\
K_1(D)Z_1(t) & \textrm{if}\ t\geq D,
\end{array}\right.
\end{equation}
so that, with this control, the control system \eqref{systfini} with input delay is written as
$$
X_1'(t) = A_1X_1(t) + \chi_{(D,+\infty)}(t)B_1 K_1(D) Z_1(t-D),
$$
with $Z_1$ given by \eqref{transfo_Artstein}.
Here the notation $\chi_E$ stands the characteristic function of $E$, that is the function defined by $\chi_{E}(t)=1$ whenever $t\in E$ and $\chi_{E}(t)=0$ otherwise.
Using \eqref{transfo_Artstein}, the feedback $\alpha$ defined by \eqref{defalpha} is such that
\begin{equation}\label{def_alpha}
\alpha(t) = \left\{ \begin{array}{ll}
0 & \textrm{if}\ t<D,\\
K_1(D) X_1(t) + K_1(D) \int_{\max(t-D,D)}^t e^{(t-D-s)A_1}B_1\alpha(s)\, ds & \textrm{if}\ t\geq D.
\end{array}\right.
\end{equation}
In other words, the value of the feedback control $\alpha$ at time $t$ depends on $X_1(t)$ and of the controls applied in the past -- more precisely, of the values of $\alpha$ over the time interval $(\max(t-D,D),t)$.

\begin{lemma}
The feedback \eqref{def_alpha} stabilizes exponentially the control system \eqref{systfini}. 
\end{lemma}

\begin{proof}
By construction $t\mapsto Z_1(t)$ converges exponentially to $0$, and hence $t\mapsto \alpha(t)$ and thus $t\mapsto \int_{\max(t-D,D)}^t e^{(t-D-s)A_1)}B_1\alpha(s)\, ds$ converge exponentially to $0$ as well. Then the equality \eqref{transfo_Artstein} implies that $t\mapsto X_1(t)$ converges exponentially to $0$.
\end{proof}

\paragraph{Inversion of the Artstein transform.}
We show here how to invert the Artstein transform, with two motivations in mind:
\begin{itemize}
\item First of all, it is interesting to express the stabilizing control $\alpha$ (defined by \eqref{defalpha}) directly as a feedback of $X_1$. 
\item Secondly, it is interesting to express the Lyapunov functional $V_1$ (defined by \eqref{defV1}) as a function of $X_1$.
\end{itemize}
To reach this objective, it suffices to solve the fixed point implicit equality \eqref{def_alpha}. For every function $f$ defined on $\R$ and locally integrable, we define 
$$
(T_Df)(t) = K_1(D) \int_{\max(t-D,D)}^t e^{(t-D-s)A_1}B_1 f(s)\, ds.
$$
It follows from \eqref{def_alpha} that $\alpha(t) = K_1(D)X_1(t)+(T_D\alpha)(t)$, for every $t\geq D$. 
We have the following lemma, proved in \cite{PrieurTrelat}.

\begin{lemma}\label{lem_alpha}
There holds
\begin{equation}\label{alphaexpanded}
\alpha(t) = \left\{ \begin{array}{ll}
0 & \textrm{if}\ t<D,\\
\displaystyle\sum_{j=0}^{+\infty} (T_D^j K_1(D)X_1)(t) & \textrm{if}\ t\geq D,
\end{array}\right.
\end{equation}
and the series is convergent, whatever the value of the delay $D\geq 0$ may be.
\end{lemma}

Note that the value of the feedback $\alpha$ at time $t$,
\begin{equation*}
\begin{split}
\alpha(t) &= K_1(D)X_1(t) + K_1(D) \int_{\max(t-D,D)}^t e^{(t-D-s)A_1}B_1K_1(D) X_1(s)\, ds \\
& \qquad
+ K_1(D) \int_{\max(t-D,D)}^t e^{(t-D-s)A_1}B_1 K_1(D) \int_{\max(s-D,D)}^s e^{(s-D-\tau)A_1}B_1K_1(D) X_1(\tau)\, d\tau \, ds \\
& \qquad
+\cdots
\end{split}
\end{equation*}
depends on the past values of $X_1$ over the time interval $(D,t)$.
Since the feedback is retarded with the delay $D$, the term $\alpha(t-D)$ appearing at the right-hand side of \eqref{systfini} only depends on the values of $X_1(s)$ with $0<s<t-D$, as desired.

We stress that in the above result the convergence of the series is the nontrivial fact. Otherwise the formula can be obtained from an immediate formal computation.
\begin{remark}\label{remark4}
It is also interesting to express $Z_1$ in function of $X_1$, that is, to invert the equality
\begin{equation}\label{X1enfonctiondeZ1}
Z_1(t) = X_1(t) + \int_{(t-D,t)\cap(D,+\infty)} e^{(t-s-D)A_1}B_1 K_1(D) Z_1(s) \, ds
\end{equation}
coming from \eqref{transfo_Artstein} and \eqref{defalpha}. 
Although it is technical and not directly useful to derive the exponential stability of $Z_1$, it will however allow us to express the Lyapunov functional $V_1$ defined by \eqref{defV1}.
Note that
\begin{equation}\label{ref_intervalle}
(t-D,t)\cap(D,+\infty) = \left\{\begin{array}{ll}
\emptyset & \textrm{if}\ t<D, \\
(D,t) & \textrm{if}\ D<t<2D, \\
(t-D,t) & \textrm{if}\ 2D<t.
\end{array}\right.
\end{equation}
In particular if $t<D$ then $Z_1(t)=X_1(t)$.
Actually we have the following precise result (see \cite{PrieurTrelat}).

\begin{lemma}\label{lemtechnique}
For every $t\in\R$, there holds
\begin{equation}\label{X1Z1}
X_1(t) = Z_1(t) - \int_{(t-D,t)\cap(D,+\infty)} f(t-s) X_1(s)\, ds,
\end{equation}
where $f$ is defined as the unique solution of the fixed point equation
$$
f(r) =  f_0(r)  + (\tilde T_D f)(r),
$$
with $f_0(r)=e^{(r-D)A_1} B_1K_1(D)$ and
$$
(\tilde T_D f)(r) = \int_0^{r} e^{(r-\tau-D)A_1} B_1K_1(D) f(\tau)\, d\tau.
$$
Moreover, we have
\begin{equation*}
\begin{split}
f(r) &= \sum_{j=0}^{+\infty} (\tilde T_D^j f_0)(r) \\
& = e^{(r-D)A_1} B_1K_1(D)  \\
&\qquad + \int_0^{r} e^{(r-\tau-D)A_1} B_1K_1(D) e^{(\tau-D)A_1} B_1K_1(D)   \, d\tau \\
&\qquad + 
\int_0^{r} e^{(r-\tau-D)A_1} B_1K_1(D) \int_0^{\tau} e^{(\tau-s-D)A_1} B_1K_1(D) e^{(s-D)A_1} B_1K_1(D)   \, ds   \, d\tau \\
&\qquad + \cdots
\end{split}
\end{equation*}
and the series is convergent, whatever the value of the delay $D\geq 0$ may be.
\end{lemma}

With this expression and using (\ref{X1enfonctiondeZ1}) in Remark \ref{remark4}, the feedback $\alpha$ can be as well written as
\begin{equation*}
\begin{split}
\alpha(t)&= \chi_{(D,+\infty)}(t) K_1(D)Z_1(t) \\
& = \chi_{(D,+\infty)}(t) K_1(D) X_1(t) + K_1(D)\int_{(t-D,t)\cap(D,+\infty)} f(t-s) X_1(s)\, ds ,
\end{split}
\end{equation*}
and we recover of course the expression \eqref{alphaexpanded} derived in Lemma \ref{lem_alpha}.
\end{remark}

Plugging this feedback into the control system \eqref{systfini} yields, for $t>D$, the closed-loop system
\begin{equation}\label{closed-loop_system_complique}
\begin{split}
X'_1(t) &=A_1X_1(t) + B_1\alpha(t-D) \\
&= A_1X_1(t) + B_1K_1(D)X_1(t-D) +B_1 K_1(D) \int_{(t-D,t)\cap(D,+\infty)} f(t-s) X_1(s)\, ds  ,
\end{split}
\end{equation}
which is, as said above, exponentially stable. Moreover, the Lyapunov function $V_1$, which is exponentially decreasing according to Remark \ref{rem_Lyap1}, can be written as
\begin{equation*}
V_1(t) = \frac{1}{2} \left(X_1(t)+ \int_{I_t(D)} f(t-s) X_1(s)\, ds\right)^\top P(D) \left(X_1(t)+ \int_{I_t(D)} f(t-s) X_1(s)\, ds\right) .
\end{equation*}
with $I_t(D) = (t-D,t)\cap(D,+\infty)$.
We stress once again that the above feedback and Lyapunov functional stabilize the system whatever the value of the delay may be.

\begin{remark}
Let us make a remark on the practical implementation.
Although the expression \eqref{alphaexpanded} has some theoretical interest, in practice we do not use it to compute $\alpha(t)$, and we use instead the gain matrix $K_1(D)$ whose computation, based on the knowledge of $(A_1,e^{-DA_1}B_1)$, is a very easy task.
Moreover, instead of considering the closed-loop system \eqref{closed-loop_system_complique}, it is far more convenient to consider the equivalent system
\begin{equation}\label{dynamstab}
\begin{split}
X_1'(t) &= AX_1(t) + B_1K_1(D)Z_1(t-D)  ,\\
Z_1'(t) &= (A_1+e^{-DA_1}B_1K_1(D))Z_1(t) ,
\end{split}
\end{equation}
which in this form looks more like a dynamic stabilization procedure (see \cite{Sontag}). These implementation issues are analyzed in detail in \cite{PrieurTrelat}.
\end{remark}

\subsection{Stabilization of the whole system}
In order to prove that the feedback $\alpha$ designed above stabilizes the whole system \eqref{sys-dim-infinie} we have to take into account the rest of the system (consisting of modes that are naturally stable). What has to be checked is whether or not the delayed control part might destabilize this infinite-dimensional part. 

Let $(u_D(\cdot),w(\cdot))$ denote a solution of \eqref{newzero} in which we choose the control $\alpha$ in the feedback form designed previously, such that $u_D(0)=0$ and $w(0)=0$.
Here, we make a slight abuse of notation, since $w(t)$ designates the solution $w(t,\cdot)\in H^2(0,L)\cap H^1_0(0,L)$ satisfying
\begin{equation}\label{eq444}
\begin{split}
& u_D' = \alpha,\quad w' = Aw+au_D+b\alpha ,\\
& u_D(0)=0, \quad w(0,\cdot)=0.
\end{split}
\end{equation}
Let $M(D)$ be a positive real number such that
\begin{equation}\label{defM}
\begin{split}
M(D)\ > \ & \Vert b \Vert_{L^2(0,L)}^2 \Vert K_1(D)\Vert_{\R^{n+1}}^2 \\
& + \max\left(2 \Vert a\Vert^2_{L^2(0,L)},  \frac{\max(\lambda_1,\ldots,\lambda_n) }{\lambda_{\min}(P(D))} \right)\max\left(1,D e^{2D\Vert A_1\Vert} \Vert B_1\Vert_{\R^{n+1}}^2 \Vert K_1(D)\Vert_{\R^{n+1}}^2\right)  ,
\end{split}
\end{equation}
where $ \Vert K_1(D)\Vert_{\R^{n+1}}^2 = \sum_{j=0}^{n} k_j(D)^2$,
$ \Vert B_1\Vert_{\R^{n+1}}^2 = 1+\sum_{j=1}^n b_j^2$, where $\Vert A_1\Vert$ is the usual matrix norm induced from the Euclidean norm of $\R^{n+1}$, and where $\lambda_{\min}(P(D))>0$ is the smallest eigenvalue of the symmetric positive definite matrix $P(D)$. The precise value of $M(D)$ is not important however. What is important in what follows is that $M(D)>0$ is large enough.

We set
\begin{equation}\label{defVD}
\begin{split}
V_D(t) &= M(D)\,V_1(t) + M(D)\, \int_{(t-D,t)\cap(D,+\infty)} V_1(s)\, ds - \frac{1}{2}\langle w(t),Aw(t)\rangle_{L^2(0,L)} \\
&= \frac{M(D)}{2}Z_1(t)^\top P(D)Z_1(t) + \frac{M(D)}{2} \int_{(t-D,t)\cap(D,+\infty)} Z_1(s)^\top P(D)Z_1(s)\, ds \\
&\qquad\qquad\qquad\qquad\qquad\qquad\qquad\qquad\qquad\qquad\qquad\qquad
 - \frac{1}{2}\sum_{j=1}^{+\infty} \lambda_jw_j(t)^2.
\end{split}
\end{equation}
We are going to prove that $V_D(t)$ is positive and decreases exponentially to $0$. This Lyapunov functional consists of three terms. The two first terms stand for the unstable finite-dimensional part of the system. As we will see, the integral term is instrumental in order to tackle the delayed terms. The third term stands for the infinite-dimensional part of the system. In this infinite sum actually all modes are involved, in particular those that are unstable. Then the two first terms of \eqref{defVD}, weighted with $M(D)>0$, can be seen as corrective terms and this weight $M(D)>0$ is chosen large enough so that $V_D(t)$ be indeed positive.
More precisely,
\begin{equation}\label{12:55}
- \frac{1}{2}\sum_{j=1}^{+\infty} \lambda_jw_j(t)^2 = - \frac{1}{2}\sum_{j=1}^n \lambda_jw_j(t)^2 - \frac{1}{2}\sum_{j=n+1}^\infty \lambda_jw_j(t)^2,
\end{equation}
where $\lambda_j\geq 0$ for every $j\in\{1,\ldots,n\}$ and $\lambda_j\leq-\eta<0$ for every $j>n$ (see \eqref{refeta}). 
Therefore the second term of \eqref{12:55} is positive and the first term, which is nonpositive, is actually compensated by the first term of $V_D(t)$ since $M(D)$ is large enough, as proved in the following more precise lemma.

\begin{lemma} \label{lemma2.4}
There exists $C_2(D)>0$ such that
\begin{equation}\label{equiv2}
V_D(t) \geq C_2(D) \left( u_D(t)^2 + \Vert w(t)\Vert_{H^1_0(0,L)}^2 \right), 
\end{equation}
for every $t\geq 0$. 
\end{lemma}

\begin{proof}
First of all, by definition of $\lambda_{\min}(P(D))$, one has
\begin{equation}\label{lambdamin1}
\begin{split}
& \frac{M(D)}{2}Z_1(t)^\top P(D)Z_1(t) + \frac{M(D)}{2} \int_{t-D}^t Z_1(s)^\top P(D)Z_1(s)\, ds \\
\geq &\ M(D) \frac{\lambda_{\min}(P(D))}{2} \left( \Vert Z_1(t)\Vert_{\R^{n+1}}^2 + \int_{t-D}^t \Vert Z_1(s)\Vert_{\R^{n+1}}^2\, ds \right) ,
\end{split}
\end{equation}
for every $t\geq 0$.
Besides, recall that, from \eqref{X1enfonctiondeZ1}, one has
$$
X_1(t) = Z_1(t) - \int_{(t-D,t)\cap(D,+\infty)} e^{(t-s-D)A_1}B_1 K_1(D) Z_1(s) \, ds,
$$
and therefore, using the Cauchy-Schwarz inequality and the inequality $(a+b)^2\leq 2a^2+2b^2$, it follows that
\begin{equation}\label{est3}
\Vert X_1(t)\Vert_{\R^{n+1}}^2 \leq C_3(D)\left( \Vert Z_1(t)\Vert_{\R^{n+1}}^2 +\int_{t-D}^t \Vert Z_1(s)\Vert_{\R^{n+1}}^2\, ds \right),
\end{equation}
with 
$$
C_3(D) = \max\left(2,2D e^{2D\Vert A_1\Vert} \Vert B_1\Vert_{\R^{n+1}}^2 \Vert K_1(D)\Vert_{\R^{n+1}}^2\right).
$$
We then infer from \eqref{lambdamin1} and \eqref{est3} that
\begin{equation}\label{lambdamin2}
\begin{split}
& \frac{M(D)}{2}Z_1(t)^\top P(D)Z_1(t) + \frac{M(D)}{2} \int_{t-D}^t Z_1(s)^\top P(D)Z_1(s)\, ds \\
\geq &\ M(D) \frac{\lambda_{\min}(P(D))}{2C_3(D)} \Vert X_1(t)\Vert_{\R^{n+1}}^2 ,
\end{split}
\end{equation}
for every $t\geq 0$.

Using \eqref{12:55} and the definition of $X_1$ in (\ref{eq:def:A1}), we have
\begin{equation}\label{13:00}
- \frac{1}{2}\sum_{j=1}^{+\infty} \lambda_jw_j(t)^2 \geq - \frac{1}{2}\sum_{j=n+1}^\infty \lambda_jw_j(t)^2 -  \frac{1}{2} \max_{1\leq j\leq n} ( \lambda_j ) \Vert X_1(t)\Vert_{\R^{n+1}}^2 ,
\end{equation}
and therefore, using \eqref{lambdamin2}, we get
$$
V_D(t) \geq \left( M(D) \frac{\lambda_{\min}(P(D))}{2C_3(D)} -  \frac{1}{2} \max_{1\leq j\leq n} ( \lambda_j ) \right) \Vert X_1(t)\Vert_{\R^{n+1}}^2  - \frac{1}{2}\sum_{j=n+1}^\infty \lambda_jw_j(t)^2 ,
$$
for every $t\geq 0$.
By definition of $M(D)$ (see \eqref{defM}), one has $M(D) \frac{\lambda_{\min}(P(D))}{2C_3(D)} -  \frac{1}{2} \max_{1\leq j\leq n} ( \lambda_j ) >0$ and hence there exists $C_4(D)>0$ such that
\begin{equation}\label{13:47}
V_D(t) \geq C_4(D) \left(  \Vert X_1(t)\Vert_{\R^{n+1}}^2  - \frac{1}{2}\sum_{j=n+1}^\infty \lambda_jw_j(t)^2 \right) .
\end{equation}

Using the series expansion $w(t,\cdot)=\sum_{i=1}^{+\infty} w_i(t)e_i(\cdot)$, we have
$$\Vert w(t)\Vert_{H^1_0(0,L)}^2=\sum_{(i,j)\in(\N^*)^2}w_i(t)w_j(t)\int_0^L e_i'(x) e_j'(x)\, dx.$$
By definition, one has $e_n''+ce_n=\lambda_ne_n$ and $e_n(0)=e_n(L)=0$, for every $n\in\N^*$. Integrating by parts and using the orthonormality property, we get
$$
\int_0^L e_i'(x) e_j'(x)\, dx =\int_0^L c(x)e_i(x)e_j(x) \, dx - \lambda_j\delta_{ij}  , 
$$
with $\delta_{ij}=1$ whenever $i=j$ and $\delta_{ij}=0$ otherwise, and thus, for all $t\geq 0$,
\begin{equation}\label{cp:1}
\Vert w(t)\Vert_{H^1_0(0,L)}^2=\int_0^L c(x)w(t,x)^2\,dx - \sum_{j=1}^\infty \lambda_j w_j(t)^2.
\end{equation}
Since $c\in L^\infty(0,L)$, it follows that
\begin{equation*}
\begin{split}
\Vert w(t)\Vert_{H^1_0(0,L)}^2 &\leq \Vert c\Vert_{L^\infty(0,L)}\ \Vert w(t)\Vert_{L^2(0,L)}^2 - \sum_{j=1}^n \lambda_j w_j(t)^2 - \sum_{j=n+1}^\infty \lambda_j w_j(t)^2 \\
&\leq \Vert c\Vert_{L^\infty(0,L)}\sum_{j=1}^\infty  w_j(t)^2  - \sum_{j=n+1}^\infty \lambda_j w_j(t)^2 \\
&\leq \Vert c\Vert_{L^\infty(0,L)}\Vert X_1(t)\Vert_{\R^{n+1}}^2  - \sum_{j=n+1}^\infty (\lambda_j- \Vert c\Vert_{L^\infty(0,L)}) w_j(t)^2 \\
\end{split}
\end{equation*}
and since $\lambda_j\rightarrow -\infty$ as $j$ tends to $+\infty$, there exists $C_5>0$ such that
$$
\Vert w(t)\Vert_{H^1_0(0,L)}^2 \leq - C_5 \left(  \Vert X_1(t)\Vert_{\R^{n+1}}^2  - \frac{1}{2}\sum_{j=n+1}^\infty \lambda_jw_j(t)^2 \right).
$$
Then \eqref{equiv2} follows from \eqref{13:47}.
\end{proof}

Using \eqref{ref_intervalle}, note that if $t<D$ then the integral term of \eqref{defVD} is equal to $0$ and $Z_1(t)=X_1(t)$, and hence
\begin{equation*}
V_D(t) = \frac{M(D)}{2}X_1(t)^\top P(D)X_1(t)  - \frac{1}{2}\sum_{j=1}^{+\infty} \lambda_jw_j(t)^2 ,
\end{equation*}
for every $t<D$. This remark leads to the following lemma.

\begin{lemma} \label{lemma2.5}
There exists $C_6(D)>0$ such that
\begin{equation}\label{equiv3}
V_D(t) \leq C_6(D) (u_D(t)^2+ \Vert w(t)\Vert_{H^1_0(0,L)}^2 ), 
\end{equation}
for every $t < D$.
\end{lemma}

\begin{proof}
Using \eqref{cp:1}, one has
$$ 
- \sum_{j=1}^{+\infty} \lambda_jw_j(t)^2
\leq \Vert w(t)\Vert_{H^1_0(0,L)}^2 + \Vert c\Vert_{L^\infty(0,L)} \ \Vert w(t)\Vert_{L^2(0,L)}^2 
\leq  C_8(D)\Vert w(t)\Vert_{H^1_0(0,L)}^2,
$$
and then the lemma follows from the Poincar\'e inequality $\Vert w(t)\Vert_{L^2(0,L)}^2 \leq L \Vert w(t)\Vert_{H^1_0(0,L)}^2$.
\end{proof}

\begin{lemma}\label{lem5}
The functional $V_D$ decreases exponentially to $0$.
\end{lemma}

\begin{proof}
Let us compute $V_D'(t)$ for $t>2D$ and state a differential inequality satisfied by $V_D$. 
First of all, it follows from \eqref{poleshifting} (in Corollary \ref{corkalman}) that
$$
\frac{d}{dt} \frac{M(D)}{2}Z_1(t)^\top P(D)Z_1(t) = -M(D) \Vert Z_1(t)\Vert_{\R^{n+1}}^2,
$$
and thus
\begin{equation*}
\begin{split}
\frac{d}{dt} \frac{M(D)}{2}\int_{t-D}^t Z_1(s)^\top P(D)Z_1(s)\, ds 
&= -M(D) \int_{t-D}^t \Vert Z_1(s)\Vert_{\R^{n+1}}^2 \, ds .
\end{split}
\end{equation*}
Then, using \eqref{eq444}, \eqref{defVD} and the fact that $A$ is self-adjoint, we get
\begin{equation}\label{diffV1}
\begin{split}
V_D'(t) =& -M(D) \Vert Z_1(t)\Vert_{\R^{n+1}}^2  - M(D) \int_{t-D}^t \Vert Z_1(s)\Vert_{\R^{n+1}}^2 \, ds \\
& -  \Vert Aw(t)\Vert_{L^2(0,L)}^2 - \langle Aw(t),a\rangle_{L^2(0,L)}u_D(t) - \langle Aw(t),b\rangle_{L^2(0,L)}K_1(D)Z_1(t) ,
\end{split}
\end{equation}
for every $t>2D$. From Young's inequality, we derive the estimates
\begin{equation}\label{est1}
\left\vert\langle Aw(t),a\rangle_{L^2(0,L)}u_D(t)\right\vert \leq \frac{1}{4}\Vert Aw(t)\Vert^2_{L^2(0,L)}+ \Vert a\Vert^2_{L^2(0,L)} \Vert X_1(t)\Vert^2_{\R^{n+1}} ,
\end{equation}
and
\begin{equation}\label{est2}
\left\vert\langle Aw(t),b\rangle_{L^2(0,L)}K_1(D)Z_1(t)\right\vert \leq \frac{1}{4}\Vert Aw(t)\Vert^2_{L^2(0,L)} + \Vert b \Vert_{L^2(0,L)}^2 \Vert K_1(D)\Vert_{\R^{n+1}}^2 \Vert Z_1(t)\Vert^2_{\R^{n+1}}.
\end{equation}
With the estimates \eqref{est1}, \eqref{est2} and \eqref{est3}, we infer from \eqref{est3} and from \eqref{diffV1} that
\begin{equation*}
\begin{split}
V_D'(t) \leq& - \left( M(D)-\Vert b \Vert_{L^2(0,L)}^2 \Vert K_1(D)\Vert_{\R^{n+1}}^2  - \Vert a\Vert^2_{L^2(0,L)} C_3(D) \right) \Vert Z_1(t)\Vert_{\R^{n+1}}^2  \\
& - \left( M(D) - \Vert a\Vert^2_{L^2(0,L)} C_3(D) \right) \int_{t-D}^t \Vert Z_1(s)\Vert_{\R^{n+1}}^2 \, ds  - \frac{1}{2} \Vert Aw(t)\Vert_{L^2(0,L)}^2 .
\end{split}
\end{equation*}
From \eqref{defM}, the real number $M(D)$ has been chose large enough so that 
$$
M(D)-\Vert b \Vert_{L^2(0,L)}^2 \Vert K_1(D)\Vert_{\R^{n+1}}^2  - \Vert a\Vert^2_{L^2(0,L)} C_3(D) >0
$$
and
$$
M(D) - \Vert a\Vert^2_{L^2(0,L)} C_3(D) >0.
$$
Therefore, there exists $C_7(D)>0$ such that
\begin{equation}\label{14:14}
V_D'(t) \leq -C_7(D) \left( \Vert Z_1(t)\Vert_{\R^{n+1}}^2 + \int_{t-D}^t \Vert Z_1(s)\Vert_{\R^{n+1}}^2 \, ds    \right)  - \frac{1}{2} \Vert Aw(t)\Vert_{L^2(0,L)}^2 .
\end{equation}

Let us provide an estimate of $\Vert Aw(t)\Vert_{L^2(0,L)}^2$. Since $-\lambda_j\leq\lambda_j^2$ as $j$ tends to $+\infty$, it follows that there exists $C_8>0$ such that
\begin{equation*}
\begin{split}
-\frac{1}{2} \langle w(t),Aw(t)\rangle_{L^2(0,L)} 
&= -\frac{1}{2}\sum_{j=1}^{n} \lambda_j w_j(t)^2 -\frac{1}{2}\sum_{j=n}^{+\infty} \lambda_j w_j(t)^2 \\
&\leq -\frac{1}{2}\sum_{j=n}^{+\infty} \lambda_j w_j(t)^2 \\
&\leq \frac{1}{2C_8} \sum_{j=1}^{+\infty} \lambda_j^2 w_j(t)^2 =  \frac{1}{2C_8} \Vert Aw\Vert_{L^2(0,L)}^2 .
\end{split}
\end{equation*}
Hence it follows from \eqref{14:14} that 
$$
V_D'(t) \leq -C_7(D) \left( \Vert Z_1(t)\Vert_{\R^{n+1}}^2 + \int_{t-D}^t \Vert Z_1(s)\Vert_{\R^{n+1}}^2 \, ds    \right)  - \frac{C_8}{2}  \langle w(t),Aw(t)\rangle_{L^2(0,L)}  .
$$
Finally, using \eqref{lambdamin1}, there exists $C_9(D)>0$ such that
$$
V_D'(t) \leq -C_9(D) V_D(t),
$$
for every $t>2D$. Therefore $V_D(t)$ decreases exponentially to $0$.
\end{proof}

From Lemma \ref{lem5}, $V_D(t)$ decreases exponentially to $0$. It follows from Lemmas \ref{lemma2.4} and \ref{lemma2.5} that there exists $C_{10}(D)>0$ and $\mu>0$ such that 
$$
u_D(t) ^2 + \Vert w(t)\Vert_{H^1_0(0,L)}^2 \leq 
C_{10}(D) e^{-\mu t} (u_D(0) ^2 + \Vert w(0)\Vert_{H^1_0(0,L)}^2 )
$$
for every $t\geq 0$. Using \eqref{defw} the proof of Theorem \ref{thm1} follows.



\begin{thebibliography}{99}

\bibitem{Artstein}
Z. Artstein,
\textit{Linear systems with delayed controls: A reduction},
IEEE Trans. Automat. Cont. {\bf 27} (1982), no. 4, 869--879.

\bibitem{BellmanCooke}
R. Bellman, K.L. Cooke, 
\textit{Differential difference equations},
New York: Academic Press, 1963.

\bibitem{CoronTrelat2004}
J.-M. Coron, E. Tr\'elat,
\textit{Global steady-state controllability of 1-D semilinear heat equations},
SIAM J. Control Optim. {\bf 43} (2004), no. 2, 549--569.

\bibitem{CoronTrelat2006}
J.-M. Coron, E. Tr\'elat,
\textit{Global steady-state stabilization and controllability of 1-D semilinear wave equations},
Commun. Contemp. Math. {\bf 8} (2006), no. 4, 535--567.

\bibitem{FridmanNicaiseValein2010}
E. Fridman, S. Nicaise, J. Valein,
\textit{Stabilization of second order evolution equations with unbounded feedback with time-dependent delay},
SIAM J. Cont. Optim. {\bf 48} (2010), no. 8, 5028--5052.

\bibitem{FridmanOrlov2009}
E. Fridman, Y. Orlov,
\textit{Exponential stability of linear distributed parameter systems with time-varying delays},
Automatica {\bf 45} (2009), 194--201.

\bibitem{FridmanShaked2002}
E. Fridman, U. Shaked,
\textit{An improved stabilization method for linear time-delay systems},
IEEE Trans. Autom. Cont. {\bf 47} (2002), no. 11, 1931--1937.

\bibitem{InfanteCastelan1978}
E.F. Infante, W.B. Castelan,
\textit{?A Liapunov functional for a matrix difference-differential equation},
J. Diff. Eq. {\bf 29} (1978), 439--451.

\bibitem{Khalil} 
H.K. Khalil,
\textit{Nonlinear systems},
Macmillan Publishing Company, New York, 1992.

\bibitem{Krstic_SCL2009}
M. Krstic,
\textit{Control of an unstable reaction–diffusion PDE with long input delay},
Syst. Cont. Letters {\bf 58} (2009), 773--782.

\bibitem{NicaisePignottiValein2011}
S. Nicaise, C. Pignotti, J. Valein,
\textit{Exponential stability of the wave equation with boundary time-varying delay},
Discrete Cont. Dynam. Syst. Ser. S {\bf 4} (2011), no. 3, 693--722.

\bibitem{NicaiseValein2010}
S. Nicaise, J. Valein,
\textit{Stabilization of second order evolution equations with unbounded feedback with time-dependent delay},
ESAIM Control Calc. Var. {\bf 16} (2010), 420--456.

\bibitem{NicaiseValeinFridman2009}
S. Nicaise, J. Valein, E. Fridman,
\textit{Stability of the heat and wave equations with boundary time-varying delays},
Discrete Cont. Dynam. Syst. Ser. S {\bf 2} (2009), no. 3, 559--581.

\bibitem{PrieurTrelat}
C. Prieur, E. Tr\'elat,
\textit{Predictor-based stabilization of finite-dimensional linear autonomous control systems with a constant delay},
Preprint Hal (2014).

\bibitem{Richard2003}
J.-P. Richard,
\textit{Time-delay systems: an overview of some recent advances and open problems},
Automatica {\bf 39} (2003), no. 10, 1667--1694.

%
%


\bibitem{Sontag}
E.D. Sontag,
\textit{Mathematical control theory. Deterministic finite-dimensional systems},
Second edition, Texts in Applied Mathematics, 6, Springer-Verlag, New York, 1998, xvi+531 pp.

\bibitem{Trelat2005}
E. Tr\'elat,
\textit{Contr\^ole optimal (French) [Optimal control], Th\'eorie \& applications [Theory and applications]},
Math. Concr\`etes [Concrete Mathematics], Vuibert, Paris, 2005.

\bibitem{ZhangKnopseTsiotras2001}
J. Zhang, C.R. Knopse, P. Tsiotras,
\textit{Stability of time-delay systems: equivalence between Lyapunov and scaled small-gain conditions},
IEEE Trans. Automat. Cont. {\bf 46} (2001), no. 3, 482--486.

\end{thebibliography}
\end{document}